\documentclass[11pt]{article}
\usepackage[T1]{fontenc}
\usepackage{epsfig}
\usepackage{algorithm}
\usepackage{algorithmic}
\usepackage{latexsym}
\usepackage{amsmath}
\usepackage{cite}
\usepackage{enumerate}
\usepackage{verbatim}
\usepackage{amsfonts}
\usepackage{graphicx}
\usepackage[normalem]{ulem}
\usepackage{array}
\usepackage{authblk}
\usepackage{color}\usepackage[dvipsnames]{xcolor}

\usepackage{tikz}\usetikzlibrary{patterns}

\usepackage[
   bookmarks=false, 
   colorlinks,
   citecolor=brown!70!black,
   linkcolor=brown!80!black,
   urlcolor=blue!70!black,
]{hyperref}
\usepackage{relsize, exscale}
\usepackage{pdfpages, microtype, bm, amssymb, amsthm}

\newtheorem{lem}{Lemma}

\newtheorem{cor}{Corollary}
\newtheorem{thm}{Theorem}

\newtheorem{conj}{Conjecture}

\def\a{0cm} 
\def\b{1cm} \def\B{1.5cm}
\def\c{2cm} \def\C{2.5cm}
\def\d{3cm}

 \def\N{13.5cm}

\def\sizePoint{4pt}

\newcommand{\point}[2]{\fill (canvas cs:x=#1,y=#2) circle (\sizePoint); }

\def\bb#1{\bm{\mathrm{#1}}}
\def\pex{\bb{pex}}
\def\cyc{\bb{cyc}}
\def\pcyc{\bb{pcyc}}

\def\exc{\bb{exc}}
\def\des{\bb{des}}
\def\fix{\bb{fix}}
\def\N{\mathbb{N}}
\def\pdes{\bb{des}}
\def\mm{\bb}

\newcommand{\stkout}[1]{\ifmmode\text{\sout{\ensuremath{#1}}}\else\sout{#1}\fi}

\title{Transformation à la Foata \\
for special kinds of descents and
excedances }
\author[]{Jean-Luc Baril and Sergey Kirgizov}

\affil[]{LIB, Universit{\'e} Bourgogne Franche-Comt{\'e}\\
B.P. 47 870, 21078 Dijon-Cedex, France\\
\href{mailto:barjl@u-bourgogne.fr,sergey.kirgizov@u-bourgogne.fr}{\tt \{barjl,sergey.kirgizov\}@u-bourgogne.fr}
}

\begin{document}
\maketitle

\begin{abstract} A pure excedance in a permutation
$\pi=\pi_1\pi_2\ldots \pi_n$ is a position $i<\pi_i$
such that  there is no $j<i$ with $i\leq \pi_j<\pi_i$.  We present a
one-to-one correspondence on the symmetric group that transports 
pure excedances to descents of special kind. As a byproduct, we prove that the
popularity of pure excedances equals those of pure descents on
permutations, while their distributions are different.

\end{abstract}

\section{Introduction and notations}
The distribution of the number of descents has been widely studied on
several classes of combinatorial objects such as permutations
\cite{Mac}, cycles \cite{Eli, Eli1}, and words \cite{Bar,Foa1}. Many
interpretations of this statistic appear in several fields as Coxeter
groups \cite{Ber,Gar} or lattice path theory \cite{Ges}. One of the
most famous result involves the {\it Foata fundamental
  transformation}~\cite{Foa} to establish a one-to-one correspondence
between descents and excedances on permutations. This bijection
provides a more straightforward proof than those of MacMahon
\cite{Mac} for the equidistribution of these two Eulerian statistics.

In this paper, we
present a bijection {\it à la Foata} on the symmetric group that exchanges pure excedances
with special kind of descents defined as a mesh pattern $p_2$~\cite{Bra} (see below for the definition of this pattern).
Then, we deduce that the popularities (but not the distributions) of
pure descents \cite{Bar0} and pure excedances are the same. This common popularity is
given by the generalized Stirling number $n!\cdot (H_n-1)$ (see
\href{https://oeis.org/A001705}{A001705} in \cite{Sloa}) where
$H_n=\sum_{k=1}^n\frac{1}{k}$ is the $n$-th harmonic number. Finally,
we conjecture the existence of a bijection on the symmetric group
that exchanges pure excedances and $p_2$ while preserving 
the number of cycles.

Let $S_n$ be the set of permutations of length $n$, {\it i.e.}, all
bijections from $[n]=\{1,2,\ldots,n\}$ into itself. The one-line
representation of a permutation $\pi\in S_n$ is $\pi=\pi_1\pi_2\ldots
\pi_n$ where $\pi_i=\pi(i)$, $1\leq i\leq n$. For $\sigma\in S_n$, the
{\it product} $\sigma\cdot \pi$ is the permutation
$\sigma(\pi_1)\sigma(\pi_2)\ldots \sigma(\pi_n)$. A $\ell${\it
-cycle} $\pi=\langle i_1,i_2,\ldots, i_\ell\rangle$ in $S_n$ is a $n$-length
permutation satisfying $\pi(i_1)=i_2, \pi(i_2)=i_3, \ldots,
\pi(i_{\ell-1})=i_\ell, \pi(i_\ell)=i_1$ and $\pi(j)=j$ for $j\in [n]\backslash
\{i_1,i_2, \ldots, i_\ell\}$. For $1\leq k\leq n$, we denote by $C_{n,k}$
the set of all $n$-length permutations admitting a decomposition in a
product of $k$ disjoint cycles. The set $C_{n,k}$ is counted by the
signless Stirling numbers of the first kind $c(n,k)$ defined by
$$c(n,k)=(n-1)\;c(n-1,k)+c(n-1,k-1)$$ where $c(n,k)=0$ if $n=0$ or
$k=0$, except $c(0,0)=1$ (see \cite{Sta,Wil} and  \href{https://oeis.org/A132393}{A132393} in
\cite{Sloa}). These numbers
also enumerate $n$-length permutations $\pi$ having $k$ {\it left-to-right maxima},
{\it i.e.}, positions $i\in [n]$ such that $\pi_j<\pi_i$ for $j<i$ (see
\cite{Sta}), and permutations $\pi\in S_n$ with $k-1$ {\it pure
descents}, {\it i.e.}, descents $\pi_i>\pi_{i+1}$ where there is no $j<i$
such that $\pi_j\in [\pi_{i+1},\pi_i]$ (see \cite{Bar0}). Note that a
pure descent can be viewed as an occurrence of the mesh pattern
$(21,L_1)$ where $L_1=\{1\}\times[0,2]\cup \{(0,1)\}$. Indeed, for a
$k$-length permutation $\sigma$ and a subset $R\subseteq
[0,k]\times[0,k]$, an occurrence of the mesh pattern $(\sigma,R)$ in a
permutation $\pi$ is an occurrence of $\sigma$ in $\pi$ with the 
additional  restriction that no element of $\pi$ lies inside the shaded
regions defined by $R$, where $(i,j)\in R$ means the square having
bottom left corner $(i,j)$  in the graphical representation $\{(i,
\sigma_i), i\in[k]\}$ of $\sigma$. For instance, an occurrence of the
mesh pattern $p_1$ in Figure \ref{fig1}  corresponds to an occurrence of
a pure descent. See \cite{Bra} for a more detailed definition of 
 mesh patterns.

  Regarding this interpretation of pure descents in terms of mesh
patterns, we define other kinds of descents by the mesh patterns
$p_i=(21,L_i)$, $p'_i=(21,R_i)$ with $L_i=\{1\}\times[0,2]\cup \{(0,i)\}$
and $R_i=\{1\}\times[0,2]\cup \{(2,i)\}$ for $0\leq i\leq 2$. Modulo the
trivial symmetries on permutations (reverse and complement), it is
straightforward to see that $p_0$, $p_1$, and $p_2$ are respectively in
the same distribution class as $p'_2$, $p'_1$ and $p'_0$. 
Then, we deal
with only mesh patterns $p_i$, $i\in[0,2]$. We refer to Figure
\ref{fig1} for a graphical illustration. On the other hand, we
define a {\it pure excedance} as an occurrence of an excedance, {\it i.e.}
$\pi_i>i$, with the additional restriction that there is no point
$(j,\pi_j)$ such that $1\leq j\leq i-1$ with $i\leq \pi_j<\pi_i$.
Although such a pattern (called $pex$) is not a mesh pattern, we can represent it
graphically as shown in Figure \ref{fig1}.

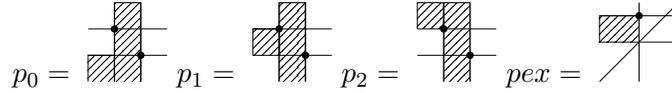
\begin{figure}[ht]
     \begin{center}
     $p_0=~$\scalebox{0.35}    {\begin{tikzpicture}
                 \draw [thick] (\b,\a) -- (\b,\d);
                 \draw [thick] (\c,\a) -- (\c,\d);
                 \draw [thick] (\a,\b) -- (\d,\b);
                 \draw [thick] (\a,\c) -- (\d,\c);


         \draw [fill=lightgray,pattern=north east
lines](\b,\a)--(\b,\d)--(\c,\d)--(\c,\a);
               \draw [fill=lightgray,pattern=north east
lines](\a,\a)--(\a,\b)--(\b,\b)--(\b,\a);
         \point{\b}{\c};\point{\c}{\b}
             \end{tikzpicture}}
             $p_1=~$\scalebox{0.35}    {\begin{tikzpicture}
                 \draw [thick] (\b,\a) -- (\b,\d);
                 \draw [thick] (\c,\a) -- (\c,\d);
                 \draw [thick] (\a,\b) -- (\d,\b);
                 \draw [thick] (\a,\c) -- (\d,\c);


         \draw [fill=lightgray,pattern=north east
lines](\b,\a)--(\b,\d)--(\c,\d)--(\c,\a);
               \draw [fill=lightgray,pattern=north east
lines](\a,\b)--(\a,\c)--(\b,\c)--(\b,\b);
         \point{\b}{\c};\point{\c}{\b}
             \end{tikzpicture}}
             $p_2=~$\scalebox{0.35}    {\begin{tikzpicture}
                 \draw [thick] (\b,\a) -- (\b,\d);
                 \draw [thick] (\c,\a) -- (\c,\d);
                 \draw [thick] (\a,\b) -- (\d,\b);
                 \draw [thick] (\a,\c) -- (\d,\c);


         \draw [fill=lightgray,pattern=north east
lines](\b,\a)--(\b,\d)--(\c,\d)--(\c,\a);
               \draw [fill=lightgray,pattern=north east
lines](\a,\c)--(\a,\d)--(\b,\d)--(\b,\c);
         \point{\b}{\c};\point{\c}{\b}
             \end{tikzpicture}}
             $pex=~$\scalebox{0.35}    {\begin{tikzpicture}
                 \draw [thick] (\B,\a) -- (\B,\d);
                 \draw [thick] (\a,\B) -- (\d,\B);
                 \draw [thick] (\a,\C) -- (\d,\C);
                 \draw [thick] (\a,\a) -- (\d,\d);
         \draw [fill=lightgray,pattern=north east
lines](\a,\B)--(\a,\C)--(\B,\C)--(\B,\B);
         \point{\B}{\C};
             \end{tikzpicture}}
                 \end{center}
     \caption{Illustration of the mesh patterns $p_0$, $p_1$, $p_2$ and
$pex$; $p_1$  and $pex$ correspond respectively to a pure descent and a pure
excedance.}
\label{fig1}
\end{figure}

A {\it statistic}  is an integer-valued function from a set
$\mathcal{A}$ of $n$-length permutations (we use the boldface to denote statistics).
For a pattern $p$, we define the pattern statistic ${\bf p}:\mathcal{A}\rightarrow\N$ where the image ${\bf p}~\pi$ of $\pi\in\mathcal{A}$ by ${\bf p}$ is the number of occurrences of $p$ in $\pi$. The {\it popularity} of $p$ in $\mathcal{A}$ is the total
number of occurrences of $p$ over all objects of $\mathcal{A}$, that is
$\sum_{a\in\mathcal{A}}{\bf p}~a$ (see \cite{Bon} for instance).
Below, we present statistics that we use throughout the paper:
$$\begin{array}{ll}
\mbox\exc~\pi&= \mbox{number of excedances in } \pi,\\
\mbox\pex~\pi&= \mbox{number of pure excedances in } \pi,\\
\mbox\des~\pi&= \mbox{number of descents in } \pi,\\
\mbox\pdes_i~\pi&= \mbox{number of patterns $p_i$ in } \pi,~0\leq i\leq 2,\\
\mbox\fix~\pi&= \mbox{number of fixed points in } \pi,\\
\mbox\cyc~\pi&= \mbox{number of cycles in the decomposition of } \pi,\\
\mbox\pcyc~\pi&= \mbox{number of pure cycles (i.e. cycles of length
at least two) in } \pi,\\
                    &= \cyc~\pi - \fix~\pi\\
\end{array}$$

 We organize the paper as follows. In Section 2, we focus on patterns $p_i$, $0\leq i\leq 2$. We prove that the
statistics $\pdes_0$ and $\pdes_1$ are equidistributed by giving
algebraic and bijective proofs. Next, we provide the bivariate exponential
generating function for the distribution of $p_2$, and we deduce that
$p_2$ has the same popularity as $p_0$ and $p_1$, without having the
same distribution. In Section 3, we present a bijection on  $S_n$ that
transports pure excedances into patterns $p_2$. Notice that the Foata's
first transformation \cite{Foa} is not a candidate for such a bijection.
As a consequence, pure descents and pure excedances are
equipopular on $S_n$, but they do not have the same distribution.
Combining all these results, we deduce that patterns $p_i$, $0\leq i\leq 2$,
and $pex$ are equipopular on the symmetric group $S_n$.
Finally we present two conjectures about the equidistribution of $(\mm{cyc},\des_2)$ and 
$(\mm{cyc}, \pex)$, and that of $(\des,\pdes_2)$ 
and $(\exc, \pex)$.

\section{The statistics $\pdes_i$, $0\leq i\leq 2$}

For $0\leq i\leq 2$, let $A^i_{n,k}$ be the set of $n$-length
permutations having $k$ occurrences of $p_i$, and 
denote by
$a^i_{n,k}$ its
cardinality. Let $A^i(x,y)$ be the bivariate exponential generating
function  $\sum_{n = 0}^\infty\sum_{k=0}^{n-1}
a^i_{n,k}\frac{x^n}{n!}y^k$. In~\cite{Bar0,Kit}, it is proved that
$a^1_{n,k}$ equals the signless Stirling numbers of the first kind
$c(n,k+1)$ (see \href{https://oeis.org/A132393}{A132393} in
\cite{Sloa}). Indeed, a permutation  $\sigma\in A^1_{n,k}$ can be
uniquely obtained from an $(n-1)$-length permutation $\pi$  by one of the
two following constructions:

\begin{enumerate}[(i)]
\item  if $\pi\in A^1_{n-1,k-1}$, then we increase by one all values of
$\pi$ greater than or equal to $\pi_{n-1}$, and we add $\pi_{n-1}$ at the end;

\item  if $\pi\in A^1_{n-1,k}$, then we increase by one all values of
$\pi$ greater than or equal to a given value $x\leq n$, $x\neq \pi_{n-1}$ and
  we add $x$ at the end.
\end{enumerate}

Then, we deduce the recurrence relation
$a^1_{n,k}=a^1_{n-1,k-1}+(n-1)a^1_{n-1,k}$ with $a^1_{n,0}=(n-1)!$ for
$n\geq 1$, $a^1_{0,0}=1$ and the bivariate exponential generating
function is
$$A^1(x,y)=\frac{1}{y(1-x)^y}-\frac{1}{y}+1$$ which proves that
$a^1_{n,k}=c(n,k+1)$.

\medskip 
Below, we prove that $a^1_{n,k}$ also counts
$n$-length permutations having $k$ occurrences of the pattern $p_0$.

  \begin{thm} The number $a_{n,k}^0$ of $n$-length permutations having
$k$ occurrences of pattern $p_0$ equals $a^1_{n,k}=c(n,k+1)$.
  \label{thm1}
  \end{thm}

  \begin{proof}
An $n$-length permutation $\sigma\in A^0_{n,k}$ can be uniquely obtained
from an $(n-1)$-length permutation $\pi$  by one of the two following
constructions:

\begin{enumerate}[(i)]
\item
  if $\pi\in A^0_{n-1,k-1}$, then we increase by one all values of
$\pi$ and we add $1$ at the end;

\item
  if $\pi\in A^0_{n-1,k}$,  then we increase by one all values of
$\pi$  greater than or equal to a given value $x$,  $1 < x \leq n$, and we add
  $x$ at the end.
\end{enumerate}

   We deduce the recurrence relation
$a^0_{n,k}=a^0_{n-1,k-1}+(n-1)a^0_{n-1,k}$ with the initial condition
   $a^0_{n,0}=(n-1)!$, and then $a^0_{n,k}=a^1_{n,k}=c(n,k+1)$.
  \end{proof}

Now, we focus on the distribution of the pattern $p_2$. Table~\ref{tab1} 
provides exact values for small sizes.

\begin{thm} We have  $$A^2(x,y)=\frac{e^{x(1-y)}}{(1-x)^y},$$ and the
general term $a^2_{n,k}$ satisfies for $n\geq 2$ and $1\leq k\leq\lfloor\frac{n}{2}\rfloor$
$$a^2_{n,k}=na^2_{n-1,k}+(n-1)a^2_{n-2,k-1}-(n-1)a^2_{n-2,k}$$ with the
initial conditions $a^2_{n,0}=1$ and $a^2_{n,k}=0$ for $n\geq 0$ and
$k>\lfloor\frac{n}{2}\rfloor$ (see Table~\ref{tab1} and the triangular table
\href{https://oeis.org/A136394}{A136394} in \cite{Sloa}).
\label{thm2}
\end{thm}

\begin{proof}
 Let $\sigma = \sigma_1 \sigma_2 \ldots \sigma_n$ denote a permutation
 of length $n$ having $k$ occurrences of pattern $p_2$. Let $u_{n,k}$
 (resp. $v_{n,k}$) be the number of such permutations satisfying
 $\sigma_n = n$ (resp. $\sigma_n < n$). Obviously, we have
 $$a^2_{n,k}=u_{n,k}+v_{n,k}.$$ A permutation $\sigma$ with $\sigma_n = n$
 can be uniquely constructed from an $(n-1)$-length
 permutation $\pi$ as $\sigma = \pi_1 \pi_2 \ldots
 \pi_{n-1} n$. No new occurrences of $p_2$ are created, and we
 obtain
 $$u_{n,k}=a^2_{n-1,k}.$$
 A permutation $\sigma$ satisfying $\sigma_n < n$ can be uniquely
 obtained from an $(n-1)$-length permutation $\pi$ by adding a value
 $x < n$ on the right side of its one-line notation, after increasing
 by one all the values greater than or equal to $x$.  This construction
 creates a new pattern $p_2$ if and only if $\pi$ ends with $n-1$.
 Thus, we deduce $$v_{n,k}=(n-1)u_{n-1,k-1}+(n-1)v_{n-1,k}.$$
 Combining the equations, we obtain for $n\geq 2$ and $k\geq 1$
$$a^2_{n,k}=na^2_{n-1,k}+(n-1)a^2_{n-2,k-1}-(n-1)a^2_{n-2,k},$$ which
implies the following differential equation
$$ \frac{ \partial A^2(x,y)}{\partial x} =(y-1)xA^2(x,y)+ \frac{
  \partial \left( xA^2(x,y) \right) }{\partial x} \text{, where}~
A^2(x,0)=1. $$ A simple calculation provides the claimed closed form
for the generating function $A^2(x,y)$.
\end{proof}

\begin{cor} For $0\leq i\leq 2$, the patterns $p_i$  are equipopular on
$S_n$. Their popularity is  given by the generalized Stirling number
$n!\cdot (H_n-1)$ (see \href{https://oeis.org/A001705}{A001705} in
\cite{Sloa}) where $H_n=\sum_{k=1}^n\frac{1}{k}$ is the $n$-th harmonic
number.
\label{cor1}
\end{cor}
\begin{proof}
The generating function of the popularity is
directly deduced from the bivariate generating function of pattern
distribution by calculating $$\left.\frac{ \partial A^1(x,y)}{\partial y}\right|_{y=1}=\left.\frac{ \partial A^2(x,y)}{\partial y}\right|_{y=1}.$$
\end{proof}

\begin{table}[ht!]
\begin{center}
                        \begin{tabular}{c|cccccccc}
                            $k \backslash n$  & 1&2 &3 &4 &5 &6 &7&8\\
                          \hline
                         0 & 1 & 1 & 1 & 1 & 1 & 1 & 1 & 1\\
                         1 &   & 1 & 5 & 20 & 84 & 409 & 2365 & 16064\\
                         2 &   &   &   & 3 & 35 & 295 & 2359 & 19670\\
                         3 &  &  &  &   & & 15  & 315 & 4480\\
                         4 &  &  &  &  & &   &   & 105\\
                         $\ldots$ & & & & & & &  & $\ldots$\\
                          \hline
                           $\sum$ & 1 & 2 & 6 & 24 & 120 & 720 & 5040 &
40320\\

                     \end{tabular}
  \end{center}
\caption{Number of $n$-length permutations having $k$ occurrences of
$p_2$ for $0\leq k\leq 4$ and $1\leq n\leq 8$.}
\label{tab1}
\end{table}

The statistic $\pdes_2$ has a different distribution from $\pdes_0$
and $\pdes_1$, but the three patterns $p_0,p_1,p_2$ have the same
popularity.  Below we present a bijection on $S_n$ that transports the
statistic $\pdes_2$ to the statistics $\pcyc=\cyc - \fix$.

\begin{thm} There is a one-to-one correspondence $\phi$ on $S_n$ such
  that for any $\pi\in S_n$, we have  $$\pdes_2~\pi=\pcyc~\phi(\pi).$$
\label{thm3}
\end{thm}
\begin{proof}
  Let $\pi$ be a permutation of length $n$ having
$k$ occurrences of $p_2$. We decompose $$\pi=B_0\pi_{i_1} A_1
B_1\pi_{i_2}A_2B_2\pi_{i_3}\ldots \pi_{i_k}A_{k}B_{k},$$ where

- $\pi_{i_1}<\pi_{i_2}<\ldots< \pi_{i_k}$ are the tops of the
occurrences of $p_2$,  {\it i.e.} values $\pi_{i_j}>\pi_{i_j+1}$ such
that there does not exist $\ell<i_j$ such that $\pi_\ell >\pi_{i_j}$,

-  $A_j$ is a maximal sequence such that all its values are lower than
$\pi_{i_j}$,

- for $0\leq j\leq k$, $B_j$  is an increasing sequence such that
$\pi_{i_j}<\min B_j$ and $\max B_j<\pi_{i_{j+1}}$.

Now we construct an $n$-length permutation $\phi(\pi)$ with $k$ pure
cycles as follows:
$$\phi(\pi)=\langle\pi_{i_1} A_1\rangle\cdot
\langle\pi_{i_2}A_2\rangle\cdots \langle\pi_{i_k}A_{k}\rangle.$$
For instance, if $\pi=1 2 5 3
4 6 8 7 9$ then $\phi(\pi)=\langle 5,3,4\rangle\cdot\langle
8,7\rangle$. The map
$\phi$ is clearly a bijection on $S_n$ such that $\pdes_2~\pi$ equals
the number of pure cycles in $\phi(\pi)$.
\end{proof}

Note that $\phi^{-1}$ is closely related to the Foata fundamental
transformation \cite{Foa}.

\section{The statistic $\pex$ of pure excedances}

In order to prove the equidistribution of $\pex$ and $\pdes_2$,
regarding Theorem~\ref{thm3}, it suffices to construct a bijection on
$S_n$ that transports pure excedances to pure cycles. Here, we first
exhibit a bijection on the set $D_n$ of $n$-length derangements
(permutations without fixed points), then we extend it to the set of
all permutations $S_n$.

Any permutation $\pi\in S_n$ is uniquely decomposed as a product of transpositions of the following form:
$$\pi=\langle t_1,1\rangle\cdot\langle t_2,2\rangle\cdots\langle
t_n,n\rangle$$ where $t_i$ are integers such that $1\leq t_i\leq
i$. The transposition array of $\pi$ is defined by $T(\pi)=t_1t_2\ldots
t_n$, which induces a bijection $\pi\longmapsto T(\pi)$ from $S_n$ to
the product set $T_n=[1]\times [2]\times\cdots\times [n]$.  By Lemma 1
from~\cite{Bar2}, the number of cycles of a permutation $\pi$ is given
by the number of fixed points in $T(\pi)$.  Moreover, it is
straightforward to check the two following properties:

- if $t_i=i$, then $\pi_i=i$ if and only if there is no number $j>i$
such that $t_j=t_i=i$;

- if $t_i=i$ and $\pi_i\neq i$, then $i$ is the minimal element of a cycle of length at least two in $\pi$.

So, we deduce the following lemma.

\begin{lem} The transposition array $t_1t_2\ldots t_n\in T_n$ corresponds to a derangement if and only if:
  {\it  $t_i=i$ $\Rightarrow$ there is a $j>i$ such that $t_j=i.$}
\label{lem1}
\end{lem}
\medskip

Given a derangement $\pi=\pi_1\pi_2\ldots \pi_n\in D_n$ and its
graphical representation $\{(i,\pi_i), i\in[n]\}$. We say that the
square $(i,j)\in [n]\times[n]$ is {\it free} if all following
conditions hold:
  \begin{enumerate}[(i)]
  \item Neither $\pi_i$ nor $i$ is a position of a pure excedance;
  \item $(i,j)$ is not on the first diagonal, {\it i.e.} $j\neq i$;
  \item there does not exist $k>i$ such that $\pi_k=j$;
  \item $j$ is not a pure excedance such that $j<i$ and $\pi^{-1}(j)<i$; 
  \item there does not exist $k<i$, with $\pi_k=j>i$  such that all values of the interval $[i,j-1]$ appear on the right of $\pi_i$ in $\pi$.
  \end{enumerate}

Whenever at least one of the statements above is not satisfied, we say
that the square $(i,j)$ is {\it unfree}. Notice that if $i$ and $\pi_i$ are not the positions of a pure excedance, then 
the square $(i,\pi_i)$ is always free. So, for a column $i$ of the graphical
representation of $\pi$ such that $i$ and $\pi_i$ are not the positions of a pure excedance, we label by $j$ the $j$th free square from
the bottom to the top. We refer to Figure~\ref{fig2} for an example of this labelling.

Now we define the map $\lambda$ from $D_n$ to the set $T_n^\bullet$ of transposition arrays of length $n$ satisfying the property of Lemma \ref{lem1}.

For a permutation $\pi=\pi_1\pi_2\ldots \pi_n\in D_n$, we label its graphical representation as defined above, and $\lambda(\pi)=\lambda_1\lambda_2\ldots\lambda_n$ is obtained as follows:

\begin{itemize}
\item if $i$ is a pure excedance in $\pi$, then we set $\lambda_i=i$ and  $\lambda_{\pi^{-1}(i)}=i$;
\item otherwise, $\lambda_i$ is the sum of the label of the free square $(i,\pi_i)$ with the number of pure excedances $k<i$ such that $\pi^{-1}(k)<i$.
\end{itemize}

For instance, if $\pi=6~8~12~5~4~7~3~2~11~1~9~10$ then we obtain
$\lambda(\pi)=1~1~2~4~4~2~1~1~9~1~9~10$ (see 
Figure~\ref{fig2}). 

Let us consider $i$, $1\leq i\leq n$. If $i$ is a pure excedance of
$\pi$, then we fix $\lambda_i=i$ and
$\lambda_{\pi^{-1}(i)}=i<\pi^{-1}(i)$. Otherwise, the square $(i,i)$
is unfree, and all squares $(i,\pi_k)$, $i+1\leq k\leq n$, are unfree,
which implies that the number of free squares in the $i$th column is
less than or equal to $i$. This means that $\lambda(\pi)$ lies in
$T_n$. 
Note that, by construction, all labeled squares do not correspond to any pure excedance.
Now let us prove that the square $(i,\pi_i)$ cannot be labeled $i$.
Indeed, if $\pi_i<i$ then the label of $(i,\pi_i)$ is necessarily at 
most $\pi_i\leq i-1$; otherwise, if $\pi_i>i$ then the fact that $i$ is 
not a pure excedance implies that there is 
$\pi_j\in[i,\pi_i-1]$ with 
$j<i$. Let us choose the lowest $j$ with this property. Using (v), the 
square $(i,j)$ is unfree, which implies that the label of $(i,\pi_i)$ is
less than or equal to $n$ minus the minimal number of unfree squares $(i,j)$ 
in column $i$, that is $n-(n-i+1)=i-1$.
Moreover, the
transposition array $\lambda(\pi)$ has exactly $\pex~\pi$ fixed
points, and for any fixed point $i$ there necessarily exists
$j=\pi^{-1}(i)>i$ such that $\lambda_j=\lambda_i=i$. This implies that
$\lambda(\pi)\in T_n^\bullet$.
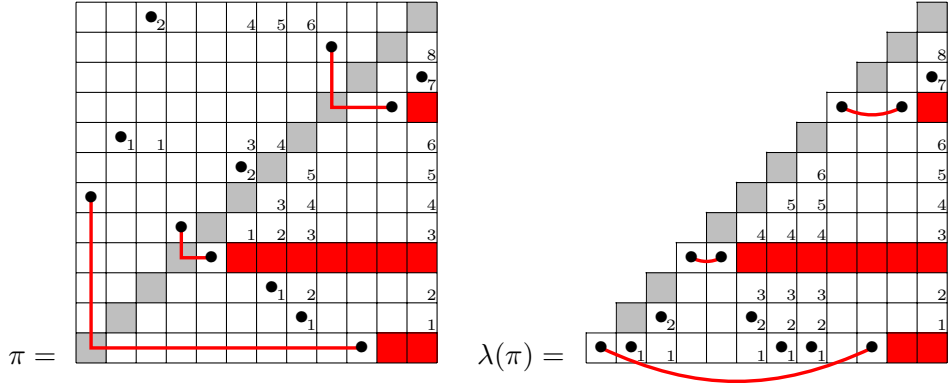
\begin{figure}[H]
	\begin{center}
$\pi =~$\begin{tikzpicture}[scale=0.4]
  \tikzstyle{r} = [fill = red, draw = none, xshift = 0.5cm, yshift = 0.5cm]  
  \tikzstyle{g} = [fill = lightgray, draw = none, xshift = 0.5cm, yshift = 0.5cm]

  \foreach \x in {0,...,11}
    \draw[g] (\x,\x) rectangle (\x + 1, \x + 1);

  \draw[r] (10,0) rectangle (12, 1);
  \draw[r] (5,3)  rectangle (12, 4);
  \draw[r] (11,8) rectangle (12, 9);
  
  \draw[step = 1, black, thin, xshift = 0.5cm, yshift = 0.5cm]
  (0, 0) grid (12, 12);


  \begin{scope}[line width = 0.45mm, red]
    \draw (1,6) -- (1,1) -- (10, 1);
    \draw (4,5) -- (4,4) -- (5, 4);
    \draw (9,11) -- (9,9) -- (11, 9);
  \end{scope}
 
  \node (p1) at (1,6) {$\bullet$};
  \node (p2) at (2,8) {$\bullet$};
  \node (p3) at (3,12) {$\bullet$};
  \node (p4) at (4,5) {$\bullet$};
  \node (p5) at (5,4) {$\bullet$};
  \node (p6) at (6,7) {$\bullet$};
  \node (p7) at (7,3) {$\bullet$};
  \node (p8) at (8,2) {$\bullet$};
  \node (p9) at (9,11) {$\bullet$};
  \node (p10) at (10,1) {$\bullet$};
  \node (p11) at (11,9) {$\bullet$};
  \node (p12) at (12,10) {$\bullet$};

  \begin{scope}[transform canvas = {xshift = 0.13cm, yshift = -0.1cm}]
    \node at (2,  8) {\tiny$1$};

    \node at (3,  8) {\tiny$1$};
    \node at (3,  12) {\tiny$2$};

    \node at (6,  5) {\tiny$1$};
    \node at (6,  7) {\tiny$2$};
    \node at (6,  8) {\tiny$3$};
    \node at (6,  12) {\tiny$4$};

    \node at (7,  3) {\tiny$1$};
    \node at (7,  5) {\tiny$2$};
    \node at (7,  6) {\tiny$3$};
    \node at (7,  8) {\tiny$4$};
    \node at (7,  12) {\tiny$5$};

    \node at (8,  2) {\tiny$1$};
    \node at (8,  3) {\tiny$2$};
    \node at (8,  5) {\tiny$3$};
    \node at (8,  6) {\tiny$4$};
    \node at (8,  7) {\tiny$5$};
    \node at (8,  12) {\tiny$6$};

    \node at (12,  2) {\tiny$1$};    
    \node at (12,  3) {\tiny$2$};
    \node at (12,  5) {\tiny$3$};
    \node at (12,  6) {\tiny$4$};
    \node at (12,  7) {\tiny$5$};
    \node at (12,  8) {\tiny$6$};
    \node at (12, 10) {\tiny$7$};
    \node at (12, 11) {\tiny$8$};

  \end{scope}

\end{tikzpicture}
\quad$\lambda(\pi)=~$\begin{tikzpicture}[scale=0.4, baseline=1.89mm]
  \tikzstyle{r} = [fill = red, draw = none, xshift = 0.5cm, yshift = 0.5cm]
  \tikzstyle{g} = [pattern = north west lines, fill = lightgray, draw = none, xshift = 0.5cm, yshift = 0.5cm]

  \draw[r] (10,0) rectangle (12, 1);
  \draw[r] (5,3)  rectangle (12, 4);
  \draw[r] (11,8) rectangle (12, 9);

  \draw[g] (1,1) rectangle (2,2);
  \draw[g] (2,2) rectangle (3,3);
  \draw[g] (4,4) rectangle (5,5);
  \draw[g] (5,5) rectangle (6,6);
  \draw[g] (6,6) rectangle (7,7);
  \draw[g] (7,7) rectangle (8,8);
  \draw[g] (9,9) rectangle (10,10);
  \draw[g] (10,10) rectangle (11,11);
  \draw[g] (11,11) rectangle (12,12);
  
  \draw[step = 1, black, thin, xshift = 0.5cm, yshift = 0.5cm]
  (0, 0) grid (12, 12);


  \begin{scope}[line width = 0.45mm, red]
    \draw (1,1) edge [bend right=25] (10, 1);
    \draw (4,4) edge [bend right=25] (5, 4);
    \draw (9,9) edge [bend right=25] (11, 9);
  \end{scope}

  \node (p1) at (1,1) {$\bullet$};
  \node (p2) at (2,1) {$\bullet$};
  \node (p3) at (3,2) {$\bullet$};
  \node (p4) at (4,4) {$\bullet$};
  \node (p5) at (5,4) {$\bullet$};
  \node (p6) at (6,2) {$\bullet$};
  \node (p7) at (7,1) {$\bullet$};
  \node (p8) at (8,1) {$\bullet$};
  \node (p9) at (9,9) {$\bullet$};
  \node (p10) at (10,1) {$\bullet$};
  \node (p11) at (11,9) {$\bullet$};
  \node (p12) at (12,10) {$\bullet$};

  \begin{scope}[transform canvas = {xshift = 0.13cm, yshift = -0.1cm}]
    \node at (2,  1) {\tiny$1$};

    \node at (3,  1) {\tiny$1$};
    \node at (3,  2) {\tiny$2$};

    \node at (6,  1) {\tiny$1$};
    \node at (6,  2) {\tiny$2$};
    \node at (6,  3) {\tiny$3$};
    \node at (6,  5) {\tiny$4$};

    \node at (7,  1) {\tiny$1$};
    \node at (7,  2) {\tiny$2$};
    \node at (7,  3) {\tiny$3$};
    \node at (7,  5) {\tiny$4$};
    \node at (7,  6) {\tiny$5$};

    \node at (8,  1) {\tiny$1$};
    \node at (8,  2) {\tiny$2$};
    \node at (8,  3) {\tiny$3$};
    \node at (8,  5) {\tiny$4$};
    \node at (8,  6) {\tiny$5$};
    \node at (8,  7) {\tiny$6$};

    \node at (12,  2) {\tiny$1$};    
    \node at (12,  3) {\tiny$2$};
    \node at (12,  5) {\tiny$3$};
    \node at (12,  6) {\tiny$4$};
    \node at (12,  7) {\tiny$5$};
    \node at (12,  8) {\tiny$6$};
    \node at (12, 10) {\tiny$7$};
    \node at (12, 11) {\tiny$8$};
  \end{scope}

  \draw[fill = white, draw = none, xshift = 0.475cm, yshift = 0.525cm]
  (-0.1, 1) -- (1, 1) --
  (1, 2) -- (2, 2) --
  (2, 3) -- (3, 3) --
  (3, 4) -- (4, 4) --
  (4, 5) -- (5, 5) --
  (5, 6) -- (6, 6) --
  (6, 7) -- (7, 7) --
  (7, 8) -- (8, 8) --
  (8, 9) -- (9, 9) --
  (9, 10) --  (10, 10) --
  (10, 11) -- (11, 11) --
  (11, 12.1) --
  (-0.1,12.1)  -- cycle;
\end{tikzpicture}

	\end{center}
        \caption{Illustration of the bijection $\lambda$
          for $\pi = 6~8~12~5~4~7~3~2~11~1~9~10$ and $\lambda(\pi)=1~1~2~4~4~2~1~1~9~1~9~10$.
        }
\label{fig2}\end{figure}

\begin{thm} The map $\lambda$ from $D_n$ to $T_n^\bullet$ is a bijection such that $$\pex~\pi=\fix ~\lambda(\pi).$$
\label{thm4}\end{thm}
\begin{proof}

Since the cardinality of $T_n^\bullet$ equals that of $D_n$, and the
image of $D_n$ by $\lambda$ is contained in $T_n^\bullet$, it suffices
to prove the injectivity.

Let $\pi$ and $\sigma$, $\pi\neq \sigma$, be two derangements in $D_n$.
If $\pi$ and $\sigma$ do not have the same pure excedances, 
then, by construction,
$\lambda(\pi)$ and $\lambda(\sigma)$ do not have the same fixed points,
and thus $\lambda(\pi)\neq \lambda(\sigma)$.

Now, let us assume that $\pi$ and $\sigma$ have the same pure
excedances. If there is a pure excedance $i$ such that 
$\pi^{-1}(i)\neq \sigma^{-1}(i)$ then the definition implies 
$\lambda(\pi)\neq\lambda(\sigma)$. Otherwise the two permutations have 
the same pure excedances $i$, and for each of them we have  
$\pi^{-1}(i)=\sigma^{-1}(i)$. Let $j$ be the greatest integer such that 
$\pi_j\neq \sigma_j$ (without loss of generality, we assume 
$\pi_j<\sigma_j$).
In this case, $j$ is not a pure excedance for the two permutations. 
Thus, $\lambda(\pi)_j$ (resp. $\lambda(\sigma)_j$) is the sum of the 
label of $(j,\pi_j)$ (resp. $(j,\sigma_j)$) with  the number of pure 
excedances $k<j$ such that $\pi^{-1}(k)<j$ (resp.  $\sigma^{-1}(k)<j$).
Since we have $\pi_j<\sigma_j$, the label of $(j,\pi_j)$ is less than 
the label of $(j,\sigma_j)$, and the number of pure excedances $k<j$ 
such that $\pi^{-1}(k)<j$ is less than or equal to the number of pure 
excedances $k<j$ such that $\sigma^{-1}(k)<j$. Then we have 
$\lambda(\pi)_j<\lambda(\sigma)_j$. Then $\lambda$ is an injective map, 
and thus a bijection.
\end{proof}

\begin{thm} There is a one-to-one correspondence $\psi$ on the set $D_n$ of $n$-length derangements such that for any $\pi\in D_n$, $$\pex~\pi=\cyc~\psi(\pi).$$
\label{thm5}
\end{thm}
\begin{proof} Considering Theorem \ref{thm3} and Theorem \ref{thm4}, we define for any $\pi\in D_n$, $\psi(\pi)=\phi(\sigma)$ where $\sigma$ is the permutation having $\lambda(\pi)$ as transposition array.
\end{proof}

\begin{thm} The two bistatistics $(\pex,\fix)$ and $(\pcyc,\fix)$ are equidistribiuted on $S_n$.
\label{thm6}
\end{thm}
\begin{proof}
  Considering Theorem \ref{thm5}, we define the map $\bar{\psi}$ on $S_n$. Let $\pi'$ be the permutation obtained from $\pi$ by deleting all fixed points and after rescaling as a permutation. Let $I=\{i_1,i_2,\ldots, i_k\}$ be the set of fixed points of $\pi$. Then, we set $\pi''=\psi(\pi')$. So, $\sigma=\bar{\psi}(\pi)$ is obtained from $\pi''$ by inserting fixed points $i\in I$ after a shift of all other entries in order to produce a permutation in $S_n$. By construction, we have $\pex~\pi= \pcyc~\sigma$ and $\fix~\pi=\fix~\sigma$ which completes the proof.
\end{proof}
A byproduct of this theorem is
\begin{cor} The statistics  $\cyc$ and  $\pex+\fix$ are equidistributed on $S_n$.
\label{cor2}
\end{cor}

Also, a direct consequence of Theorems \ref{thm3} and  \ref{thm6} is
\begin{thm} The two statistics $\pex$ and $\des_2$ are equidistributed on $S_n$.
\label{thm7}
\end{thm}

Notice that the Foata’s first transformation is not a candidate for proving the equidistribution 
of $\pex$ and $\des_2$, while it transports $\exc$ to $\des$.  
Combining Theorem~\ref{thm7} and Corollary~\ref{cor1} we have the following. 

\begin{cor} For $0\leq i\leq 2$, the patterns $p_i$ and $pex$ are equipopular on
$S_n$ (see \href{https://oeis.org/A001705}{A001705} in
\cite{Sloa}).
\label{cor3}
\end{cor}
Finally, we present two conjectures for future works.

\begin{conj} The two bistatistics  $(\des_2,\mm{cyc})$ and 
$(\pex, \mm{cyc})$ are equidistributed on $S_n$.
\label{c1}
\end{conj}


\begin{conj}  The two bistatistics  $(\des_2,\des)$ 
and $(\pex,\exc)$ are equidistributed on $S_n$.
\label{c3}
\end{conj}

It is interesting to remark that $(\des,\cyc)$ 
and $(\exc,\cyc)$ are not equidistributed. Indeed, there are 3 permutations in $S_3$ having $\exc = 1$
and $\cyc = 2$, namely $132$, $213$, $321$, but only 2
permutations with $\des = 1$ and $\cyc = 2$, 
videlicet $132$ and $213$. So, if the Conjectures~\ref{c1} and~\ref{c3} are true then their proofs are probably independent.


\section*{Acknowledgements} 
We would like to greatly thank Vincent Vajnovszki for having offered us Conjecture~\ref{c3} and the anonymous referees  for their helpful comments and suggestions.

\end{document}